\documentclass[11pt,reqno]{amsart}
\usepackage{amscd,amssymb,amsmath,amsthm}
\usepackage[arrow,matrix]{xy}
\usepackage{graphicx}
\usepackage{cite}
\newtheorem{thm}[subsection]{Theorem}
\newtheorem{lemma}[subsection]{Lemma}
\newtheorem{pro}[subsection]{Proposition}

\newtheorem{defn}[subsection]{Definition}

\numberwithin{equation}{section} 

\newcommand{\G}{{\mathcal G}}

\newcommand{\s}{{\sigma}}
\newcommand{\de}{{\xi}}

\newcommand{\bea}{\begin{eqnarray}}
\newcommand{\eea}{\end{eqnarray}}

\newcommand{\R}{\mathbb{R}}

\begin{document}
\title[On Gibbs measure of a model]{Uniqueness of Gibbs Measure for Models With Uncountable Set of Spin Values on a Cayley Tree}

\author{Yu. Kh. Eshkabilov, F. H. Haydarov, U. A. Rozikov}

 \address{Yu.\ Kh.\ Eshkabilov\\ National University of Uzbekistan,
Tashkent, Uzbekistan.}
\email {yusup62@mail.ru}

\address{F.\ H.\ Haydarov\\ National University of Uzbekistan,
Tashkent, Uzbekistan.}

 \address{U.\ A.\ Rozikov\\ Institute of mathematics and information technologies,
Tashkent, Uzbekistan.}
\email {rozikovu@yandex.ru}

\begin{abstract} We consider models with nearest-neighbor
interactions and with the set $[0,1]$ of spin values, on a Cayley
tree of order $k\geq 1$.
 It is known that the "splitting Gibbs measures" of the model
can be described by solutions of a nonlinear integral
equation. For arbitrary $k\geq 2$ we find a sufficient condition under which the integral equation
has unique solution, hence under the condition the corresponding model has unique splitting Gibbs
measure.
\end{abstract}
\maketitle



\section{Introduction} \label{sec:intro}

In this paper we consider models (Hamiltonians) with a nearest neighbor
interaction and uncountably many spin values on a Cayley tree.

One of the central problems in the theory of Gibbs measures is to
describe infinite-volume (or limiting) Gibbs measures
corresponding to a given Hamiltonian. The existence of such
measures for a wide class of Hamiltonians was established in the
ground-breaking work of Dobrushin (see, e.g. \cite{12}). However, a
complete analysis of the set of limiting Gibbs measures for a
specific Hamiltonian is often a difficult problem.

There are several papers devoted to models on Cayley trees, see for example
\cite{1}-\cite{5},\cite{p1a}, \cite{7}, \cite{8}, \cite{11}-\cite{p3}, \cite{13}, \cite{14}, \cite{16}. All these works devoted
to models with a finite set of spin values. These models have the following common property:
The  existence  of finitely many translation-invariant  and uncountable
numbers  of the  non-translation-invariant extreme Gibbs
measures. Also for several models (see, for example, \cite{p1, p1a, p2,p3}) it were proved that there exist
three  periodic  Gibbs  measures (which are invariant with respect to normal  subgroups  of  finite index of the group representation of Cayley tree) and there are
uncountable number of non-periodic Gibbs measures.

 In \cite{6} the Potts model with
a countable set of spin values on a Cayley tree is considered and it was showed
that the set of translation-invariant splitting Gibbs measures of the model contains at most one
point, independently on parameters of the Potts model with countable set of spin values
on Cayley tree. This is a crucial difference from the models with a finite set of spin values, since the last ones may have more than one translation-invariant Gibbs measures.

How "rich" is the set of translation-invariant Gibbs measures for models
with an uncountable spin values? In \cite{re} models with nearest-neighbor
interactions and with the set $[0,1]$ of spin values, on a Cayley
tree of order $k\geq 1$ are considered and we reduced the problem of describing
the "splitting Gibbs measures" of the model
to the description of the  solutions of  some nonlinear integral
equation. For $k=1$ we showed that the integral equation has a
unique solution. In case $k\geq 2$ some models (with the set
$[0,1]$ of spin values) which have a unique splitting Gibbs
measure are constructed.
In this paper we continue this investigations and give a sufficient condition on
Hamiltonian of the model with an uncountable set of spin values under which the model has unique
translation-invariant splitting Gibbs measure. But we have not any example of model
(with uncountable spin values) with more than one translation-invariant Gibbs measure.
So this is still an open problem to find such a model.

\section{Preliminaries}

A Cayley tree $\G^k=(V,L)$ of order $k\geq 1$ is an infinite
homogeneous tree (see \cite{1}), i.e., a graph without cycles, with
exactly $k+1$ edges incident to each vertices. Here $V$ is the set
of vertices and $L$ that of edges (arcs).

Consider models where the spin takes values in the set $[0,1]$, and is assigned to the vertexes
of the tree. For $A\subset V$ a configuration $\s_A$ on $A$ is an arbitrary function $\s_A:A\to
[0,1]$. Denote $\Omega_A=[0,1]^A$ the set of all configurations on $A$. A configuration $\sigma$ on
$V$ is then defined as a function $x\in V\mapsto\sigma (x)\in [0,1]$; the set of all configurations
is $[0,1]^V$. The (formal) Hamiltonian of the model is :
\begin{equation}\label{e1}
 H(\sigma)=-J\sum_{\langle x,y\rangle\in L}
\de_{\sigma(x)\sigma(y)},
\end{equation}
where $J \in R\setminus \{0\}$
and $\de: (u,v)\in [0,1]^2\to \de_{uv}\in R$ is a given bounded,
measurable function. As usually, $\langle x,y\rangle$ stands for
nearest neighbor vertices.

Let $\lambda$ be the Lebesgue measure on $[0,1]$.  On the set of all
configurations on $A$ the a priori measure $\lambda_A$ is introduced as
the $|A|$fold product of the measure $\lambda$. Here and further on
$|A|$ denotes the cardinality of $A$.   We consider a standard
sigma-algebra ${\mathcal B}$ of subsets of $\Omega=[0,1]^V$ generated
by the measurable cylinder subsets.
 A probability measure $\mu$ on $(\Omega,{\mathcal B})$
is called a Gibbs measure (with Hamiltonian $H$) if it satisfies the DLR equation, namely for any
$n=1,2,\ldots$ and $\sigma_n\in\Omega_{V_n}$:
$$\mu\left(\left\{\sigma\in\Omega :\;
\sigma\big|_{V_n}=\sigma_n\right\}\right)= \int_{\Omega}\mu ({\rm
d}\omega)\nu^{V_n}_{\omega|_{W_{n+1}}} (\sigma_n),$$ where $\nu^{V_n}_{\omega|_{W_{n+1}}}$ is the
conditional Gibbs density
$$ \nu^{V_n}_{\omega|_{W_{n+1}}}(\sigma_n)=\frac{1}{Z_n\left(
\omega\big|_{W_{n+1}}\right)}\exp\;\left(-\beta H
\left(\sigma_n\,||\,\omega\big|_{W_{n+1}}\right)\right),
$$
and $\beta={1\over T}$, $T>0 $ is temperature.
Here and below, $W_l$ stands for a `sphere' and $V_l$ for a
`ball' on the tree, of radius $l=1,2,\ldots$,
centered at a fixed vertex $x^0$ (an origin):
$$W_l=\{x\in V: d(x,x^0)=l\},\;\;V_l=\{x\in V: d(x,x^0)\leq l\};$$
and
$$L_n=\{\langle x,y\rangle\in L: x,y\in V_n\};$$
distance $d(x,y)$, $x,y\in V$, is the length of (i.e. the number of edges in) the shortest path
connecting $x$ with $y$. $\Omega_{V_n}$ is the set of configurations in $V_n$ (and $\Omega_{W_n}$
that in $W_n$; see below). Furthermore, $\sigma\big|_{V_n}$ and $\omega\big|_{W_{n+1}}$ denote the
restrictions of configurations $\sigma,\omega\in\Omega$ to $V_n$ and $W_{n+1}$, respectively. Next,
$\sigma_n:\;x\in V_n\mapsto \sigma_n(x)$ is a configuration in $V_n$ and
$H\left(\sigma_n\,||\,\omega\big|_{W_{n+1}}\right)$ is defined as the sum
$H\left(\sigma_n\right)+U\left(\sigma_n, \omega\big|_{W_{n+1}}\right)$ where
$$H\left(\sigma_n\right)
=-J\sum_{\langle x,y\rangle\in L_n}\de_{\sigma_n(x)\sigma_n(y)},$$
$$U\left(\sigma_n,
\omega\big|_{W_{n+1}}\right)=
-J\sum_{\langle x,y\rangle:\;x\in V_n, y\in W_{n+1}}
\de_{\sigma_n(x)\omega (y)}.$$
Finally, $Z_n\left(\omega\big|_{W_{n+1}}\right)$
stands for the partition function in $V_n$, with
the boundary condition $\omega\big|_{W_{n+1}}$:
$$Z_n\left(\omega\big|_{W_{n+1}}\right)=
\int_{\Omega_{V_n}} \exp\;\left(-\beta H \left({\widetilde\sigma}_n\,||\,\omega
\big|_{W_{n+1}}\right)\right)\lambda_{V_n}(d{\widetilde\sigma}_n).$$

Due to the nearest-neighbor character of the interaction, the
Gibbs measure possesses a natural Markov property: for given a
configuration $\omega_n$ on $W_n$, random configurations in
$V_{n-1}$ (i.e., `inside' $W_n$) and in $V\setminus V_{n+1}$
(i.e., `outside' $W_n$) are conditionally independent.

We use a standard definition of a translation-invariant measure
(see, e.g., \cite{12}).
 The main object of study in this
paper are translation-invariant Gibbs measures for the model (\ref{e1})
on Cayley tree. In \cite{re} this problem of description of such
measures was reduced to the description of the solutions of a nonlinear
integral equation. For finite and countable sets of spin values
this argument is well known (see, e.g.
\cite{2}-\cite{6},\cite{11},\cite{13},\cite{14},\cite{16}).

Write $x<y$ if the path from $x^0$ to $y$ goes through $x$. Call vertex $y$ a direct successor of
$x$ if $y>x$ and $x,y$ are nearest neighbors. Denote by $S(x)$ the set of direct successors of $x$.
Observe that any vertex $x\ne x^0$ has $k$ direct successors and $x^0$ has $k+1$.

Let $h:\;x\in V\mapsto h_x=(h_{t,x}, t\in [0,1]) \in R^{[0,1]}$ be mapping of $x\in V\setminus
\{x^0\}$ with $|h_{t,x}|<C$ where $C$ is a constant which does not depend on $t$.  Given
$n=1,2,\ldots$, consider the probability distribution $\mu^{(n)}$ on $\Omega_{V_n}$ defined by
\begin{equation}\label{e2}
\mu^{(n)}(\sigma_n)=Z_n^{-1}\exp\left(-\beta H(\sigma_n)
+\sum_{x\in W_n}h_{\sigma(x),x}\right),
\end{equation}
 Here, as before, $\sigma_n:x\in V_n\mapsto
\sigma(x)$ and $Z_n$ is the corresponding partition function:
\begin{equation}\label{e3} Z_n=\int_{\Omega_{V_n}}
\exp\left(-\beta H({\widetilde\sigma}_n) +\sum_{x\in W_n}h_{{\widetilde\sigma}(x),x}\right)
\lambda_{V_n}({d\widetilde\s_n}).
\end{equation}

The probability distributions $\mu^{(n)}$ are compatible if for any $n\geq 1$ and
$\sigma_{n-1}\in\Omega_{V_{n-1}}$:
\begin{equation}\label{e4}
\int_{\Omega_{W_n}}\mu^{(n)}(\sigma_{n-1}\vee\omega_n)\lambda_{W_n}(d(\omega_n))=
\mu^{(n-1)}(\sigma_{n-1}).
\end{equation} Here
$\sigma_{n-1}\vee\omega_n\in\Omega_{V_n}$ is the concatenation of
$\sigma_{n-1}$ and $\omega_n$. In this case there exists a unique
measure $\mu$ on $\Omega_V$ such that, for any $n$ and
$\sigma_n\in\Omega_{V_n}$, $\mu \left(\left\{\sigma
\Big|_{V_n}=\sigma_n\right\}\right)=\mu^{(n)}(\sigma_n)$.

\begin{defn} The measure $\mu$ is called {\it splitting
Gibbs measure} corresponding to Hamiltonian (\ref{e1}) and function
$x\mapsto h_x$, $x\neq x^0$. \end{defn}

 The following
statement describes conditions on $h_x$ guaranteeing compatibility
of the corresponding distributions $\mu^{(n)}(\sigma_n).$

 \begin{pro}\label{p1}\cite{re} {\it The probability distributions
$\mu^{(n)}(\sigma_n)$, $n=1,2,\ldots$, in}
(\ref{e2}) {\sl are compatible iff for any $x\in V\setminus\{x^0\}$
the following equation holds:
\begin{equation}\label{e5}
 f(t,x)=\prod_{y\in S(x)}{\int_0^1\exp(J\beta\de_{tu})f(u,y)du \over \int_0^1\exp(J\beta{\de_{0u}})f(u,y)du}.
 \end{equation}
Here, and below  $f(t,x)=\exp(h_{t,x}-h_{0,x}), \ t\in [0,1]$ and
$du=\lambda(du)$ is the Lebesgue measure.}
\end{pro}

From Proposition \ref{p1} it follows that for any $h=\{h_x\in R^{[0,1]},\
\ x\in V\}$ satisfying (\ref{e5}) there exists a unique Gibbs measure
$\mu$ and vice versa. However, the analysis of solutions to (\ref{e5}) is
not easy. This difficulty depends on the given function $\xi$. In
the next sections we will give a condition on such
function under which the corresponding integral
equation has unique solution.

\section{ Uniqueness of translational - invariant solution of (\ref{e5})}

In this section we consider $\xi_{tu}$ as a continuous function
and we are going to fund a condition on $\xi_{tu}$ under which the equation (\ref{e5})
has unique solution in the class of
translational-invariant functions $f(t,x)$, i.e $f(t,x)=f(t),$ for
any $x\in V$. For such functions equation (\ref{e5}) can be written as
\begin{equation}\label{e10}
f(t)=\left({\int_0^1K(t,u)f(u)du\over \int_0^1 K(0,u)f(u)du}\right)^k,
\end{equation}
where $K(t,u)=\exp(J\beta \xi_{tu})>0, f(t)>0, t,u\in [0,1].$

We put
$$C^+[0,1]=\{f\in C[0,1]: f(x)\geq 0\}.$$
We are interested to positive continuous solutions to (\ref{e10}), i.e. such that

$f\in C_0^+[0,1]=\{f\in C[0,1]: f(x)\geq 0\}\setminus \{\theta\equiv 0\}$.

Note that equation (\ref{e10}) is not linear for any $k\geq 1$.

Define the linear operator $W:C[0,1]\to C[0,1]$ by
\begin{equation}\label{e11}
(Wf)(t)=\int^1_0K(t,u)f(u)du
\end{equation} and
defined the linear functional $\omega:C[0,1]\to R$ by
\begin{equation}\label{w}
\omega(f)\equiv (Wf)(0)=\int^1_0K(0,u)f(u)du.
\end{equation}

Then equation (\ref{e10}) can be written as
\begin{equation}\label{e12}
f(t)=(A_kf)(t)=((Bf)(t))^k,
\end{equation}
where
\begin{equation}\label{B}
(Bf)(t)={(Wf)(t)\over (Wf)(0)}, \ f\in C_0^+[0,1], \ k\geq 1.
\end{equation}

\subsection{Existence of solutions to the nonlinear equation (\ref{e12})}

In \cite{re} for $k=1$ we have proved that the equation (\ref{e12}) has unique solution
for arbitrary $K(\cdot,\cdot)\in C^+[0,1]^2$ and $f(\cdot)\in C^+[0,1]$.
But for $k\geq 2$ the uniqueness is not proved yet.
Denote
$$\mathcal F_k=\left\{f\in C^+[0,1]: f(t)\geq \left(m\over M_0\right)^k\right\}, k\in \mathbb{N},$$
where
$$m=\min_{t,u\in [0,1]}K(t,u), \ \ M_0=\max_{u\in [0,1]}K(0,u).$$

It is easy to see that $\mathcal F_k$ is a closed and convex subset of $C[0,1]$. Moreover this set is
invariant with respect to operator $A_k$, i.e. $A_k(\mathcal F_k)\subset \mathcal F_k$.

\begin{pro}\label{pr1} The operator $A_k$ is continuous on $\mathcal F_k$ for any $k\geq 2$.
\end{pro}
\begin{proof} For arbitrary $C>0$ we denote
$$\mathcal F_0=\left\{f\in C^+[0,1]: f(t)\geq C, \, \forall t\in [0,1] \right\}.$$
It is obvious that the operator
$A_1$ is continuous on the set $\mathcal F_0$ (see Lemma 2 in \cite{re}).

Let $f\in \mathcal F_k$ be an arbitrary element and $\{f_n\}\subset \mathcal F_k$ such that $\lim_{n\to\infty}f_n=f$.
Since the operator $A_1$ is continuous we have $\lim_{n\to\infty}A_1f_n=A_1f$. Consequently,
there exists $C_1>0$ such that $\|A_1f_n\|\leq C_1$ for $n\in \mathbb{N}$.
Moreover we have
$$(A_1f)(t)\leq C_2={M\over m_0}, \, t\in [0,1],$$
where
$$M=\max_{t,u\in [0,1]}K(t,u), \, m_0=\min_{u\in[0,1]}K(0,u).$$
We have
\begin{equation}\label{ee2}
A_kf_n-A_kf=(Bf_n)^k-(Bf)^k=q_{k,n}(t)(A_1f_n-A_1f),
\end{equation}
where
$$q_{k,n}(t)=\sum_{j=0}^{k-1}(A_1f_n)^{k-j-1}(t)(A_1f)^j(t)>0, \, t\in [0,1].$$
Consequently,
$$q_{k,n}(t)\leq C=\sum_{j=0}^{k-1}(C_1)^{k-j-1}(C_2)^j, \, t\in [0,1].$$
Hence
$$\|A_kf_n-A_kf\|\leq C\|A_1f_n-A_1f\|,\, n\in\mathbb{N}.$$
Since $A_1$ is a continuous from the last inequality it follows that $A_k$ is continuous on $\mathcal F_k$.
\end{proof}
Denote
$$\mathcal F_k^0=\left\{f\in C^+[0,1]: \left({m\over M_0}\right)^k\leq f(t)\leq \left({M\over m_0}\right)^k\right\}.$$

\begin{pro}\label{pr2} Let $k\geq 2$. If $f\in C_0^+[0,1]$ is a solution of the equation $A_kf=f$,
then $f\in \mathcal F_k^0$.
\end{pro}
\begin{proof} Straightforward.\end{proof}

\begin{pro}\label{pr3} Let $k\geq 2$. The set $A_k(\mathcal F_k^0)$ is relatively compact in $C[0,1]$.
\end{pro}
\begin{proof}

 By Arzel\'a-Askoli's theorem (see \cite{15}, ch.III,\S 3)
it suffices to prove that the set of functions $A_k(\mathcal F^0_k)$ is equi-continuous and there exists $\gamma>0$ such that
$$h(t)\leq \gamma, \ \ \forall t\in [0,1]\ \ \mbox{and} \ \ \forall h\in A_k(\mathcal F_k^0).$$

Let $h\in A_k(\mathcal F^0_k)$ be an arbitrary function, we have
$$0<h(t)\leq \left(M\over m_0\right)^k$$
and there exists a function $f\in \mathcal F^0_k$ such that $h=A_kf$.

Now we shall prove that $A_k(\mathcal F^0_k)$ is equi-continuous. For arbitrary $t, t'\in [0,1]$ we have ($h=A_kf$)
$$|h(t)-h(t')|=|(A_1f)^k(t)-(A_1f)^k(t')|=$$ $$ \sum_{j=0}^{k-1}(A_1f)^{k-j-1}(t)(A_1f)^j(t')|(A_1f)(t)-(A_1f)(t')|\leq$$ $$ k \left(M\over m_0\right)^{k-1}{1\over
\omega(f)}\int^1_0|K(t,u)-K(t',u)|f(u)du\leq$$ $$
k \left(M\over m_0\right)^{2k-1}{1\over
\omega(f)}\int^1_0|K(t,u)-K(t',u)|du,
$$ where $\omega(f)$ is defined in (\ref{w}).

We have
$$\omega(f)\geq m_0\cdot\left(m\over M_0\right)^k, \, f\in\mathcal F_k^0.$$

Consequently,
$$|h(t)-h(t')|\leq
{k\over m_0}\left(M_0\over m\right)^k \left(M\over m_0\right)^{2k-1}\int^1_0|K(t,u)-K(t',u)|du.
$$

Since the
kernel $K(t,u)$ is  uniformly continuous on $[0,1]^2$, we conclude
that $A_k(\mathcal F^0_k)$ also is equi-continuous.
\end{proof}

By Propositions \ref{pr1}-\ref{pr3} and Schauder's theorem (see \cite{9}, p.20) one gets
the following

\begin{thm}\label{t1} The equation $A_kf=f$ has at least one
solution in $C_0^+[0,1]$ and the set of all solutions of the equation is a subset in $\mathcal F_k^0$.
\end{thm}

\subsection{The Hammerstein's nonlinear equation}

For every $k\in \mathbb N$ we consider an integral operator $H_k$ acting in
$C^+[0,1]$ as follows:

$$(H_kf)(t)=\int_0^1K(t,u)f^k(u)du.$$
If $k\geq 2$ then the operator $H_k$ is a nonlinear operator
which is called Hammerstein's operator of order $k$.
Moreover the linear operator equation $H_1f=f$ has a unique positive solution $f$
in $C[0,1]$ (see \cite{k}, p.80).

For a nonlinear homogeneous operator $A$ it is known that if there is one positive eigenfunction of
the operator $A$ then the number of the positive eigenfunctions is continuum (see \cite{k}, p.186).

Denote
$$\mathcal M_0=\left\{f\in C^+[0,1]: f(0)=1\right\}.$$

\begin{lemma}\label{l1}
The equation
\begin{equation}\label{ee3}
A_kf=f, \ k\geq 2
\end{equation}
 has a strongly positive solution iff the equation
\begin{equation}\label{ee4}
H_kf=\lambda f, \ k\geq 2
\end{equation}
has a strongly positive solution in $\mathcal M_0$.
\end{lemma}
\begin{proof} {\sl Necessariness.} Let $f_0\in C_0^+[0,1]$ be a solution of the equation (\ref{ee3}).
We have
$$(Wf_0)(t)=\omega(f_0)\sqrt[k]{f_0(t)}.$$
From this equality we get
$$(H_kh)(t)=\lambda_0h(t),$$
where $h(t)=\sqrt[k]{f_0(t)}$ and $\lambda_0=\omega(f_0)>0$.

It is easy to see that $h\in \mathcal M_0$ and $h(t)$ is an eigenfunction of the Hammerstein's operator $H_k$,
corresponding the positive eigenvalue $\lambda_0$.

{\sl Sufficiency.}  Let $k\geq 2$ and $h\in \mathcal M_0$ be an eigenfunction of the Hammerstein's operator. Then there is a
number $\lambda_0>0$ such that $H_kh=\lambda_0h$. From $h(0)=1$ we get $\lambda_0=(H_kh)(0)=\omega(h^k)$. Then

$$h(t)={H_kh\over \omega(h^k)}.$$
From this equality we get $A_kf_0=f_0$ with $f_0=h^k\in C_0^+[0,1]$.
This completes the proof. \end{proof}

\begin{thm}\label{t2} If $k\geq 2$ then every number $\lambda>0$ is an eigenvalue of the Hammerstein's operator $H_k$.
\end{thm}
\begin{proof} By Theorem \ref{t1} and Lemma \ref{l1} there exist $\lambda_0>0$ and $f_0\in \mathcal M_0$ such that
$$H_kf_0=\lambda_0f_0.$$
Take $\lambda\in (0,+\infty)$, $\lambda\ne\lambda_0$. Define function $h_0(t)\in C_0^+[0,1]$ by
$$h_0(t)=\sqrt[k-1]{\lambda\over \lambda_0}f_0(t), \ \ t\in [0,1].$$ Then
$$H_kh_0=H_k\left(\sqrt[k-1]{\lambda\over \lambda_0}f_0\right)=\lambda h_0.$$
This completes the proof. \end{proof}

Denote
$$\mathcal K=\left\{f\in C^+[0,1]: M \cdot\min_{t\in [0,1]}f(t)\geq m\cdot\max_{t\in [0,1]}f(t)\right\},$$
$$\mathcal P_k=\left\{f\in C[0,1]: {m\over M} \cdot \left(1\over M\right)^{1\over k-1}\leq f(t)\leq {M\over m}\cdot\left(1\over m\right)^{1\over k-1}\right\}, \, k\geq 2. $$

\begin{pro}\label{pr4} Let $k\geq 2$.

a) The following holds

$$H_k(C^+[0,1])\subset \mathcal K.$$

b) If a function $f_0\in C_0^+[0,1]$ is a solution of the equation
\begin{equation}\label{ee5} H_kf=f
\end{equation}
then $f_0\in \mathcal P_k$.
\end{pro}
\begin{proof} a) Let $h\in H_k(C^+[0,1])$ be an arbitrary function. Then there exists a function $f\in C^+[0,1]$
such that $h=H_kf$. Since $h$ is continuous on $[0,1]$, there are $t_1,t_2\in [0,1]$ such that
$$h_{\min}=\min_{t\in[0,1]}h(t)=h(t_1)=(H_kf)(t_1),$$
$$h_{\max}=\max_{t\in[0,1]}h(t)=h(t_2)=(H_kf)(t_2).$$
Hence
$$h_{\min}\geq m\int^1_0f^k(u)du\geq m\int^1_0{K(t_2,u)\over M}f^k(u)du={m\over M}h_{max},$$
i.e. $h\in \mathcal K$.

b) Let $f\in C_0^+[0,1]$ be a solution of the equation (\ref{ee5}).
Then we have
$\|f\|\leq M\|f\|^k$. Consequently,
$$\|f\|\geq \left(1\over M\right)^{1\over k-1}.$$
By the property a) we have
$$f(t)\geq f_{\min}=\min_{t\in [0,1]}f(t)\geq {m\over M}\|f\|.$$
Then we obtain
$$f(t)\geq {m\over M}\left({1\over M}\right)^{1\over k-1}.$$
Also we have
$$f(t)=(H_kf)(t)\geq m\int^1_0f^k(u)du\geq m f_{\min}^k.$$
Then $f_{\min}\geq m f_{\min}^k$, i.e.
$$f_{\min}\leq \left({1\over m}\right)^{1\over k-1}.$$
Hence be the property a) we get
$$f(t)\leq f_{\max}\leq {M\over m}f_{\min}\leq {M\over m}\left({1\over m}\right)^{1\over k-1}.$$ Thus we have
$f\in \mathcal P_k$.
\end{proof}
\subsection{The uniqueness of fixed point of the operators $A_k$ and $H_k$}

Now we shall prove that $A_kf=f$ and $H_kf=f$ have a unique solution in
$C_0^+[0,1]$.

\begin{lemma}\label{l2}
Assume function $f\in C[0,1]$ changes its sign on $[0,1]$. Then for every $a\in \R$ the following inequality holds
$$\|f_a\|\geq {1\over n+1}\|f\|, \ \ n\in \mathbb{N},$$
where $f_a=f_a(t)=f(t)-a, \, t\in [0,1].$
\end{lemma}
\begin{proof}
By conditions of lemma there are $t_1,t_2\in [0,1]$ such that
$$f_{\min}=f(t_1)<0, \ \ f_{\max}=f(t_2)>0.$$
In case $a=0$ the proof is obvious. We assume $a>0$

a) Let $|f_{\min}|\geq f_{\max}$. Then $\|f\|=|f_{\min}|=|f(t_1)|.$
Hence
$$\|f_a\|=\max\{|f(t_1)-a|,|f(t_2)-a|\}=|f(t_1)-a|>|f(t_1)|=\|f\|\geq {1\over n+1}\|f\|, \, n\in \mathbb{N}.$$

b) Let $|f_{\min}|<f_{\max}$ and ${1\over 2}\|f\|\geq a$. Then $\|f\|=f_{\max}=f(t_2)$ and $\|f\|-a\geq a>0$. Consequently, $$\|f_a\|=\max\{|f(t_1)-a|, |f(t_2)-a|\}\geq |f(t_2)-a|=\|f\|-a\geq {1\over 2}\|f\|\geq {1\over n+1}\|f\|, \, n\in \mathbb{N}.$$

c) Let $|f_{\min}|<f_{\max}$ and ${1\over 2}\|f\|< a$. Then $\|f\|=f(t_2)$ and
 $$\|f_a\|=\max\{|f(t_1)-a|, |f(t_2)-a|\}\geq |f(t_1)-a|>a>{1\over 2}\|f\|\geq {1\over n+1}\|f\|, \, n\in \mathbb{N}.$$
 Thus for $a>0$ the proof is completed. For $a<0$ we put $g_a(t)=g(t)-a'$ with $g(t)=-f(t)$ and $a'=-a>0$. Then
 $$\|f_a\|=\|g_a\|\geq{1\over n+1}\|g\|= {1\over n+1}\|f\|, \, n\in\mathbb{N}.$$
This completes the proof.\end{proof}

\begin{thm}\label{t3}
Let $k\geq 2$. If the kernel $K(t,u)$ satisfies the condition
\begin{equation}\label{I}
\left(M\over m\right)^k-\left(m\over M\right)^k<{1\over k},
\end{equation}
then the operator $H_k$ has a unique fixed point in $C_0^+[0,1]$.
\end{thm}
\begin{proof}

By Theorem \ref{t2} the Hammerstein's equation $H_kf=f$ has at least one solution.
 Assume that there are two solutions
$f_1\in C_0^+[0,1]$ and $f_2\in C_0^+[0,1]$, i.e $H_kf_i=f_i$,
$i=1,2$. Denote $f(t)=f_1(t)-f_2(t)$. Then by Theorem 46.6 of \cite{kz}
the function $f(t)$ changes its sign on $[0,1]$.
From Lemma \ref{l2} we get
$$\max_{t\in[0,1]}\left|f(t)-{k\over 2}(\gamma_1+\gamma_2)\int^1_0f(s)ds\right|\geq{1\over 2}\|f\|,$$
where
$$\gamma_1=\left(m\over M\right)^k, \ \ \gamma_2=\left(M\over m\right)^k.$$
By a mean value Theorem we have
$$f(t)=\int^1_0K(t,u)k\xi^{k-1}(u)f(u)du,$$
here $\xi\in C^+[0,1]$ and
$$\min\{f_1(t), f_2(t)\}\leq \xi(t)\leq \max\{f_1(t), f_2(t)\}, \, t\in [0,1].$$

By Proposition \ref{pr4} we have $\xi\in \mathcal P_k$, i.e.
$${m\over M}\left(1\over M\right)^{1\over k-1}\leq \xi(t)\leq {M\over m}\left(1\over m\right)^{1\over k-1}, \, t\in [0,1].$$
Hence
$$\gamma_1\leq K(t,u)\xi^{k-1}(u)\leq \gamma_2,\, t,u\in[0,1].$$
Therefore
$$\left|k\cdot K(t,u)\xi^{k-1}(u)-{\gamma_1+\gamma_2\over 2}\right|\leq {\gamma_2-\gamma_1\over 2}.$$
Then
\begin{equation}\label{ee6}
\left|f(t)-{k\over 2}(\gamma_1+\gamma_2)\int^1_0f(u)du\right|\leq {k\over 2}
(\gamma_2-\gamma_1)\|f\|.
\end{equation}
Assume the kernel $K(t,u)$ satisfies the condition (\ref{I}). Then $k(\gamma_2-\gamma_1)<1$ and the inequality
(\ref{ee6}) contradicts to Lemma \ref{l2}. This completes the proof.
\end{proof}
\begin{thm}\label{t4}
Let $k\geq 2$. If the kernel $K(t,u)$ satisfies the condition
(\ref{I}), then for every $\lambda>0$ the Hammerstein's equation $H_kf=\lambda f$ has unique solution in $C_0^+[0,1]$.
\end{thm}
\begin{proof}
Clearly the equation $H_kf=\lambda f$ is equivalent to the following equation
\begin{equation}\label{ee7}
\int^1_0K_\lambda(t,u)f^k(u)du=f(t),
\end{equation}
where $K_\lambda(t,u)={1\over\lambda}K(t,u)$. The kernel  $K_\lambda(t,u)$ satisfies
the condition (\ref{I}) with $\tilde{m}={m\over\lambda}$ and $\tilde{M}={M\over\lambda}$. Consequently, by Theorem \ref{t3}
it follows that the equation (\ref{ee7}) has unique solution in $C_0^+[0,1]$.
\end{proof}
\begin{thm}\label{t5}
Let $k\geq 2$. If the kernel $K(t,u)$ satisfies the condition
(\ref{I}), then the equation $A_kf=f$ has unique solution in $C_0^+[0,1]$.
\end{thm}
\begin{proof}
Assume there are two solutions $f_1,f_2\in C^+[0,1]$, $f_1\ne f_2$, i.e. $A_kf_i=f_i$, $i=1,2$.
 By Lemma \ref{l1} the functions $h_i(t)=\sqrt[k]{f_i(t)}$, $t\in[0,1]$ are solutions of the Hammerstein's equation, i.e.
 $$H_kh_i=\lambda_ih_i,\ \ i=1,2,$$
 where $\lambda_i=\omega(f_i)>0$ and $h_i\in\mathcal M_0$. On the other hand Theorem \ref{t4} implies that $\lambda_1\ne\lambda_2$.
 Let $h_0(t)\in C^+[0,1]$ be a fixed point of the Hammerstein's operator $H_k$. Then by Theorems \ref{t2} and \ref{t4} we get
 $$h_i=\sqrt[k-1]{\lambda_i}h_0(t),\, i=1,2.$$ Consequently,
 $${f_1(t)\over f_2(t)}=\gamma^k, \ \ \mbox{with} \ \ \gamma=\sqrt[k-1]{\lambda_1\over\lambda_2}.$$
 Using this equality we obtain
 $$f_1(t)=(A_kf_1)(t)=A_k(\gamma^kf_2)=A_kf_2(t)=f_2(t).$$ This completes the proof.
\end{proof}

Consider the following Hamiltonian 
\begin{equation}\label{e11}
 H(\sigma)=-J\sum_{\langle x,y\rangle\in L}
\de_{\sigma(x)\sigma(y)}=-\sum_{\langle x,y\rangle\in L}\ln K(\sigma(x),\sigma(y)),
\end{equation}
where $J \in R\setminus \{0\}$
and $K(t,u)$ satisfies the condition (\ref{I}). Then as a corollary of 
Proposition \ref{p1} and Theorem \ref{t5}  we get the following

\begin{thm}\label{t6}
Let $k\geq 2$. If the function $K(t,u)$ of the Hamiltonian (\ref{e11}) satisfies the condition
(\ref{I}), then the model (\ref{e11}) has unique translational invariant Gibbs measure.
\end{thm}
 
{\bf Example.} It is easy to see that the condition (\ref{I}) is satisfied iff
$${M\over m}\leq \eta_k=\sqrt[k]{{1+\sqrt{4k^2+1}\over 2k}}, \, k\geq 2.$$
Consider the following function
\begin{equation}\label{K1}
K(t,u)=\sum_{i=1}^m\sum_{j=1}^nc_{ij}t^iu^j+a, \, c_{ij}\geq 0, \, a>0.
\end{equation}
For this function we have $m=a$, $M=\sum_{i=1}^m\sum_{j=1}^nc_{ij}+a$.
The following is obvious

a) If  ${1\over a}\sum_{i=1}^m\sum_{j=1}^nc_{ij}\leq \eta_k-1$ then for function (\ref{K1}) the condition (\ref{I})
is satisfied.

b) If  ${1\over a}\sum_{i=1}^m\sum_{j=1}^nc_{ij}> \eta_k-1$ then for function (\ref{K1}) the condition (\ref{I})
is not satisfied.

{\it Remark.} Is there a kernel $K(t,u)>0$ of the equation (\ref{e10})
when the equation has more than one solutions? This is still open
problem.

\section*{ Acknowledgements}

UAR thanks the TWAS  Research Grant: 09-009 RG/Maths/As-I; UNESCO FR: 3240230333. He also thanks the Department of Algebra, University of Santiago de Compostela, Spain,  for providing financial support to
his visit to the Department.

\end{document}